\documentclass[11pt]{amsart} %,draft
\usepackage[margin=2.9cm]{geometry}
\usepackage{amsthm}
\usepackage{amsmath}
\usepackage{amssymb}
\usepackage{extarrows}
\usepackage{enumerate}
\usepackage{graphicx}
\usepackage[font={small}]{caption}
\usepackage{subcaption}
\usepackage{xcolor}

\usepackage{hyperref}
\hypersetup{colorlinks=true,linktoc=all,urlcolor={blue!80!black},citecolor={green!60!black}}

\newtheorem{theorem}{Theorem}
\newtheorem{lem}[theorem]{Lemma}
\newtheorem{prop}[theorem]{Proposition}

\newtheorem*{theorem*}{Theorem}
\newtheorem*{lem*}{Lemma}
\newtheorem*{prop*}{Proposition}
\newtheorem*{cor*}{Corollary}

\theoremstyle{definition}

\newtheorem*{rem*}{Remark}
\newtheorem*{rems*}{Remarks}
\newtheorem*{quest*}{Question}
\newtheorem*{notat*}{Notation}

\newtheorem*{ex*}{Example}
\newtheorem*{defin*}{Definition}

\newenvironment{tequation}{% write text as an equation
	\begin{equation}
		\begin{minipage}[center]{0.9\linewidth}}{%
		\end{minipage}
	\end{equation}\ignorespacesafterend}

\newenvironment{tequation*}{% write text as an equation*
	\begin{equation*}
		\begin{minipage}[center]{0.9\linewidth}}{%
		\end{minipage}
	\end{equation*}\ignorespacesafterend}

\def \R{{\mathbb{R}}}

\def \cH{{\mathcal{H}}}

\DeclareMathOperator{\dist}{dist}

\newcommand{\0}[1]{\overline{#1}}
\newcommand{\1}[1]{\widetilde{#1}}
\newcommand{\2}[1]{{}_{|#1}}

\newcommand{\inline}[1]{\quad\text{#1}\quad}
\newcommand{\afterline}[1]{\qquad\text{#1}\ }
\newcommand{\var}{\,\cdot\,}
\newcommand{\such}{\, : \,\,}

\let\xiff\xLeftrightarrow

\let\oldin\in
\DeclareRobustCommand{\in}{\oldin\nolinebreak[4]}
\let\indentedrule\hrulefill 
\renewcommand{\hrulefill}{\noindent\indentedrule}

\begin{document}

\title[Geometry of planar curves]{Geometry of planar curves intersecting many lines in a few points}

\author{D.~Vardakis}%{Dimitris Vardakis}
\address{Department of Mathematics, Michigan State University, East Lansing, MI. 48823}
\email{jimvardakis@gmail.com}

\author{A.~Volberg}%{Alexander Volberg}
\address{Department of Mathematics, Michigan State University, East Lansing, MI. 48823;\hfill\break\indent
		Hausdorff Center for Mathematics, Bonn, Germany}
\email{volberg@msu.edu}

%\makeatletter
%\@namedef{subjclassname@2010}{%
	%\textup{2010} Mathematics Subject Classification}
%\makeatother
%\subjclass[2010]{42B20, 42B35, 47A30}


%\date{15 марта 2021 г.}

%\tableofcontents

\begin{abstract}
	The local Lipschitz property is shown for the graph avoiding multiple point intersection with lines directed in a given cone. The assumption is much stronger than those of Marstrand's well-known theorem, but the conclusion is much stronger too. Additionally, a continuous curve with a similar property is $\sigma$-finite with respect to Hausdorff length and an estimate on the Hausdorff measure of each ``piece'' is found.
\end{abstract}

\maketitle

\section{The statement of the problem}

The problem at hand is to better understand the structure of Borel sets in~$\R^2$ that have a small intersection with parallel shifts of lines from a whole cone. Here, we work only with sets that are graphs and continuous curves. So we have strong assumptions. But the results claim some estimate on the Hausdorff measure (not merely the Hausdorff dimension).

Initially, we show that a function's graph intersecting all parallel shifts of lines from a nondegenerate cone  in at most two points is locally Lipschitz and also present a counter-example showing this fails if more intersection points are allowed.

Next, we prove that any curve that has finitely many intersections with a cone of lines is $\sigma$-finite with respect to Hausdorff length and we find a bound on the Hausdorff measure of each ``piece.''

On the other hand, in \cite{EL} it was shown that, given countably many graphs of functions, there is another function whose graph has only one intersection with all shifts of the given graphs but whose graph has dimension $2$.

This result shows that there is a ``thick'' graph having only one intersection with all shifts of countably many other graphs. In our turn, we show that the graph having finitely many intersection with shifts of the whole cone of linear functions must be in fact very ``thin''.

\begin{prop} \label{2D}
	Let $\lambda>0$ be a fixed number and consider all the cones of lines with slopes between $\lambda$ and $-\lambda$ (containing the vertical line). If $f\colon(0,1)\to\R$ is a continuous function such that any line of these cones intersects its graph at at most two points, then $f$ is locally Lipschitz.
\end{prop}

\begin{figure}%[h!]
	\centering
	\includegraphics[scale=0.15]{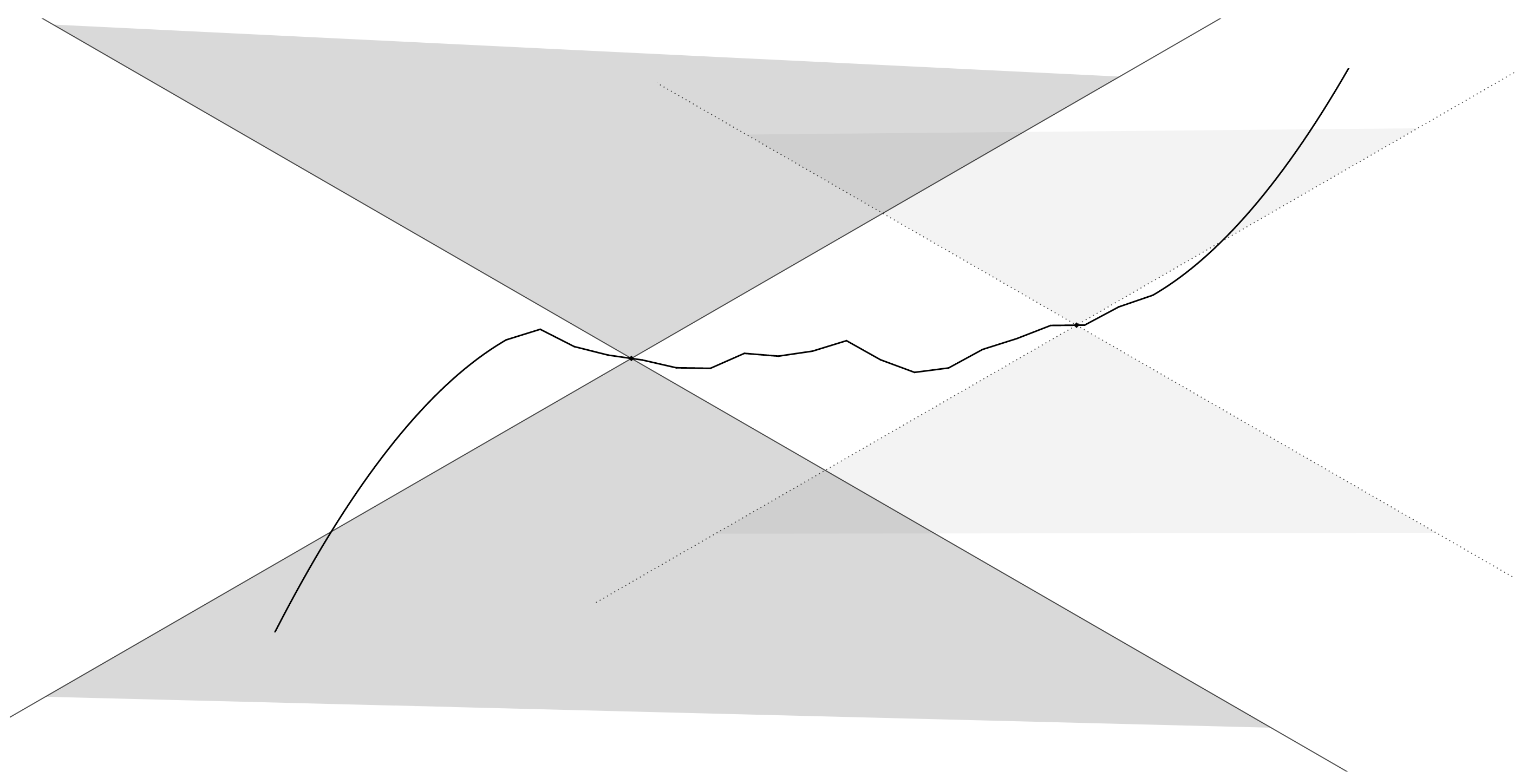}
	\caption{Each line from any cone intersects the graph at at most two points.}
	\label{cone}
\end{figure}

Notice that our hypothesis implies that no three points of the graph of $f$ can lie on the same line that is a parallel shift of a line from a given cone.

For the proof we will need the following lemmas.

\begin{lem} \label{lip}
	Every convex (or concave) function on an open interval is locally Lipschitz.
\end{lem}

\begin{lem} \label{mono}
	If a function $g\colon(0,1)\to\R$ is continuous and has a unique local extremum, $\1x,$ inside $(0,1),$ then it is strictly monotone in $(0,\1x]$ and $[\1x,1)$ with opposite monotonicity on each interval.
\end{lem}

\begin{proof}[Proof of Lemma \ref{mono}]
Suppose $\1x$ is a local minimum for $g$. We will show that~$g$ is strictly monotone increasing in $[\1x,1)$. Assume the contrary, i.e., consider two points $x_1<x_2\in[\1x,1)$ such that $g(x_1)\geq g(x_2)$. On the compact interval $[x_1,x_2]$, the function~$g$ has to attain a minimum and a maximum, which respectively are at $x_1$ and $x_2$ otherwise the uniqueness of $\1x$ is contradicted. If $x_1=\1x$, the point $\1x$ is not a local minimum and so $\1x<x_1$. Again, $\1x$ and $x_1$ must be the minimum and maximum, respectively, of $g$ in $[\1x,x_1]$, which in turn says $x_1$ is a local maximum contradicting the uniqueness of $\1x$. Therefore, $g(x_1)<g(x_2)$ and $g$ is strictly monotone increasing on $[\1x,1)$. Similarly, on $(0,\1x]$ $g$ is (strictly) monotone decreasing and the same arguments work for when $\1x$ is local maximum.
\end{proof}

\begin{proof}
Consider the slope function of $f$, $S(x,y)=\frac{f(x)-f(y)}{x-y}$, and note that
\[S(x,y)=\frac{f(x)-f(y)}{x-y}=\zeta \iff f(x)-\zeta x=f(y)-\zeta y.\]
If for any two points $x<y\in(0,1)$ we have $|S(x,y)|<\lambda$, then $f$ is Lipschitz (with Lipschitz constant at most $\lambda$).

Now suppose that there exist $x_0,y_0\in(0,1)$ for which $|S(x_0,y_0)|\geq \lambda$ and consider the case where $S(x_0,y_0)=\lambda'\geq \lambda$. Since $S(x,y)=S(y,x)$, we may assume that $x_0<y_0$. We will denote the line passing through $(x_0,f(x_0))$ and $(y_0,f(y_0))$ by $\epsilon_{\lambda'}$.

\begin{figure}%[h!]
	\centering
	\includegraphics[scale=0.15]{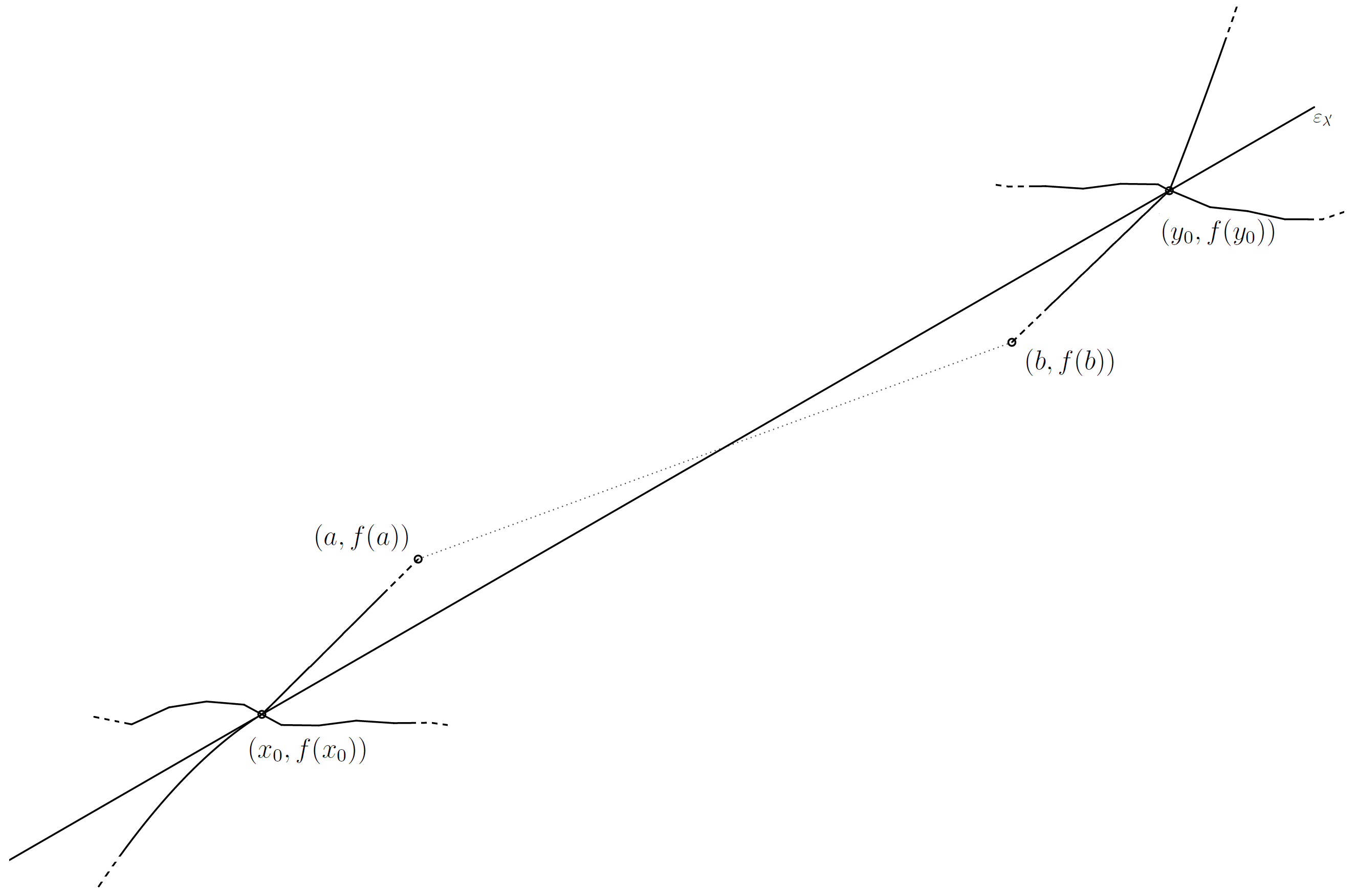}
	\caption{$S(x_0,y_0)=\lambda'\geq \lambda$; The part of the graph of $f$ between $x_0$ and $y_0$ cannot lie on different sides of $\epsilon_{\lambda'}$.}
	\label{constantly_slope}
\end{figure}

If there are numbers $x_0<a<b<y_0$ such that 
\[(S(x_0,a)-\lambda')(S(x_0,b)-\lambda')\leq0,\]
then by the continuity of $S(x,\var)$ there has to exist a number $c\in[a,b]$ such that $\frac{f(x_0)-f(c)}{x_0-c}=\lambda'=\frac{f(x_0)-f(y_0)}{x_0-y_0}$. But this means that $(x_0,f(x_0)), (c,f(c))$ and $(y_0,f(y_0))$ are colinear, which contradicts our hypothesis and therefore $S(x_0,y)$ has to be constantly greater or constantly less than $\lambda'$ for $x_0<y<y_0$ (see Figure~\ref{constantly_slope}). For the same reasons $S(x_0,y)$ has to be constantly greater or constantly less than $\lambda'$ also for $y>y_0$ and the same holds for $S(x,y_0)$ for $x<x_0$.

Graphically, this means that $\epsilon_{\lambda'}$ separates $f$ in three parts that do not intersect $\epsilon_{\lambda'}$; one before $x_0$, one over $(x_0,y_0)$, and one after $y_0$. We proceed to show that the part over $(x_0,y_0)$ lies on a different side of $\epsilon_{\lambda'}$ from the other two.

\begin{figure}%[h!]
	\begin{minipage}{0.5\linewidth}
		\includegraphics[width=\linewidth, height=3.5cm]{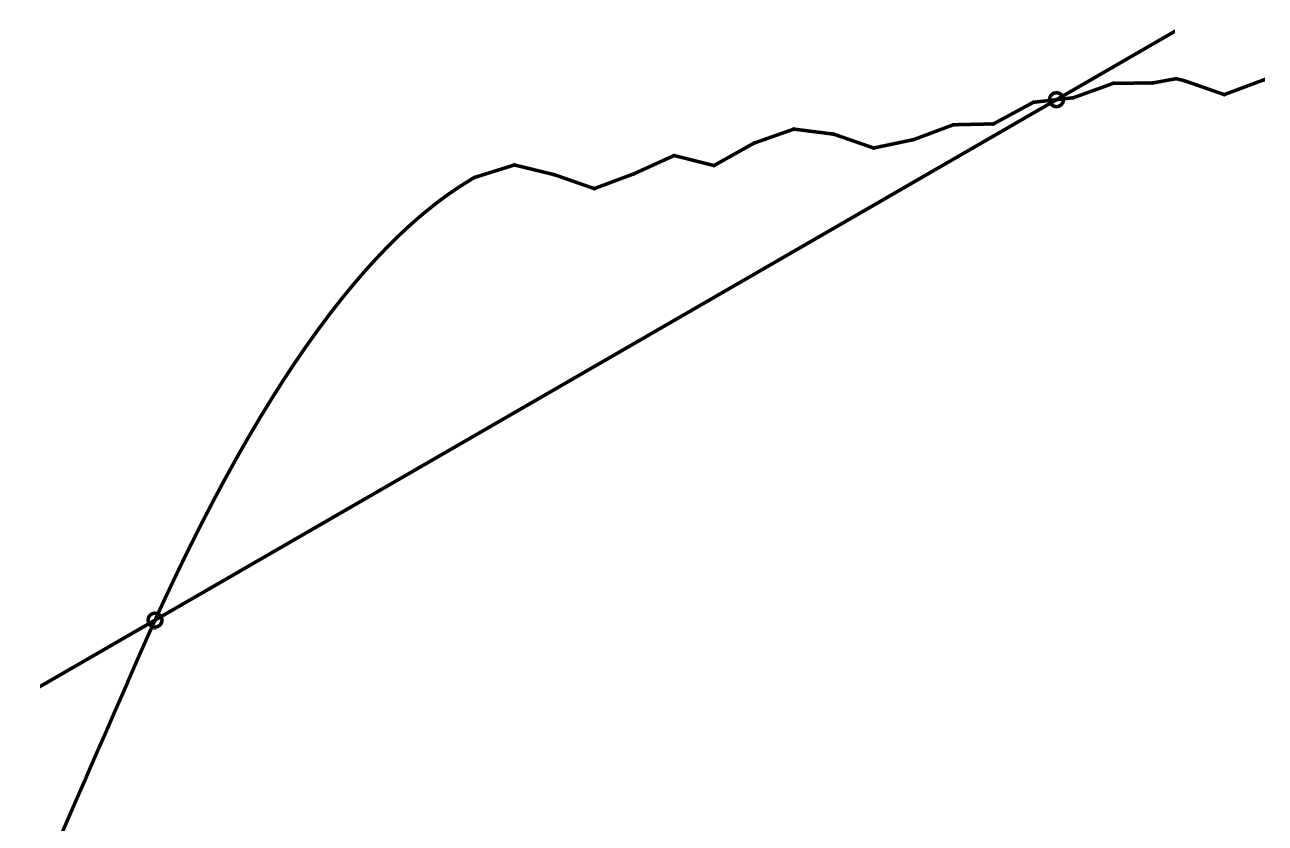}
		\subcaption*{$S(x_0,y)> \lambda'$}
	\end{minipage}
	\begin{minipage}{0.5\linewidth}
		\includegraphics[width=\linewidth, height=3.5cm]{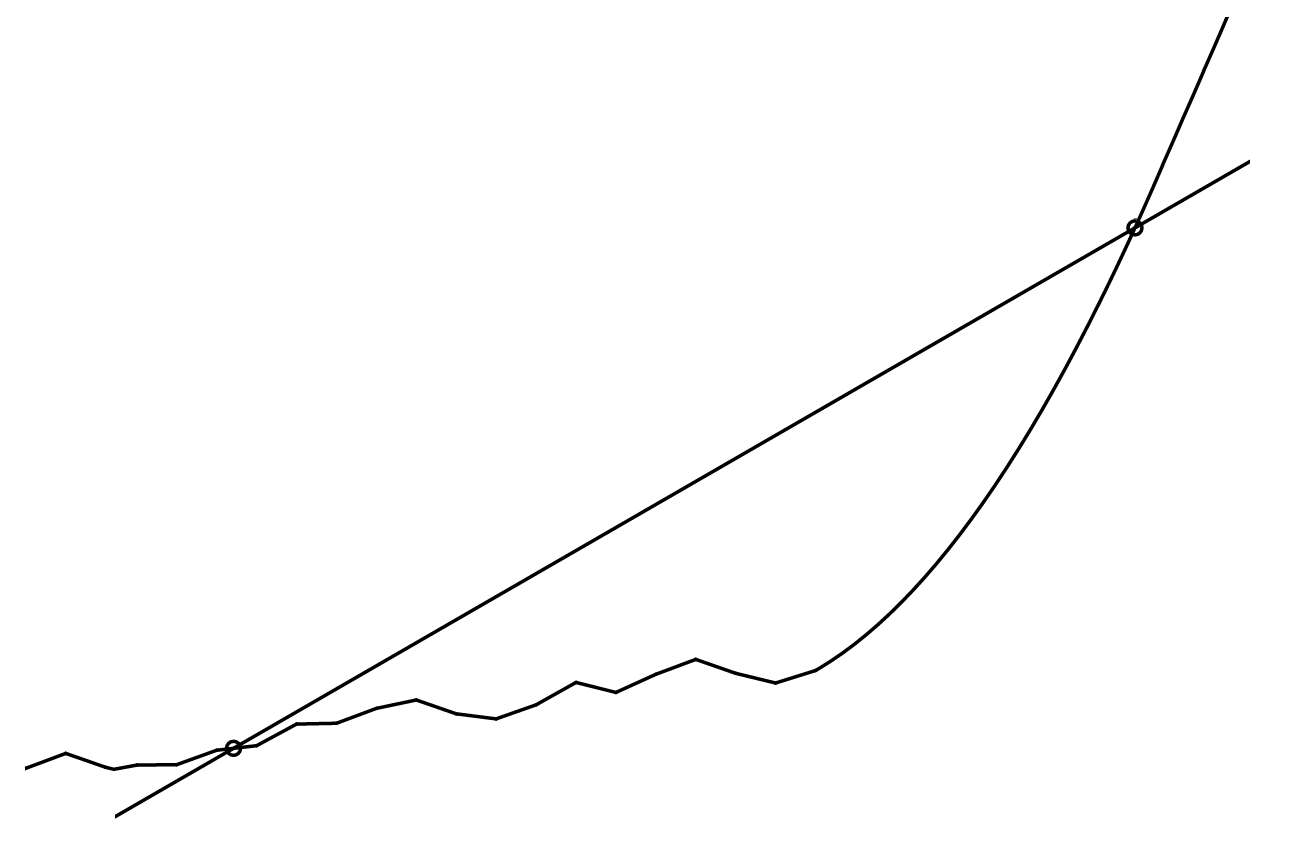}
		\subcaption*{$S(x_0,y)< \lambda'$}
	\end{minipage}
	\caption{The two cases when $x_0<y<y_0$.}
	\label{cases}
\end{figure}

Let us consider the case when $S(x_0,y)<\lambda'$ for $x_0<y<y_0$. Then, the function $f(x)-\lambda'x$ defined on $[x_0,y_0]$ attains a maximum at $x_0$ and at $y_0$ (which also implies that $S(x,y_0)>\lambda'$ for $x_0<x<y_0$) and let $\1y\in(x_0,y_0)$ be the point where $f(x)-\lambda'x$ attains a minimum (see Figure~\ref{no_going_back}). Now, suppose additionally that $S(x_0,y)<\lambda'$ also for $y>y_0$.

\begin{figure}%[h!]
	\centering
	\includegraphics[scale=0.3]{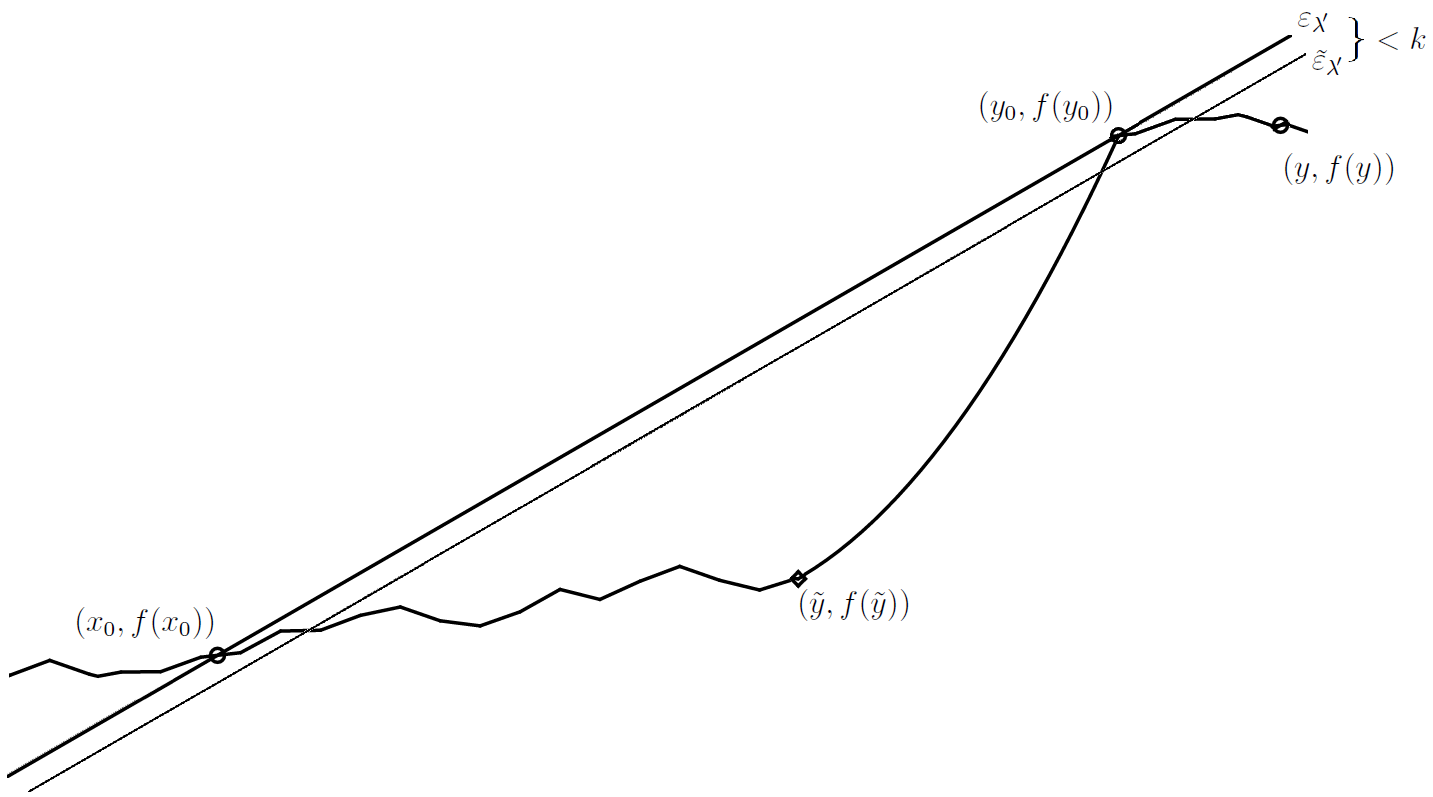}
	\caption{If $S(x_0,y)<\lambda'$ for every $y\notin(x_0,1)\setminus\{y_0\}$, by moving the line $\epsilon_{\lambda'}$ slightly down, we get three points of intersection.}
	\label{no_going_back}
\end{figure}

Pick a number $k$ with $f(x_0)-\lambda'x_0>k>\max\{f(\1y)-\lambda'\1y,f(y)-\lambda'y\}$ for some $y>y_0$. Then, we have simultaneously
\begin{align*}
f(\1y)-\lambda'\1y<k<f(x_0)-\lambda'x_0,\\
f(\1y)-\lambda'\1y<k<f(y_0)-\lambda'y_0,\\
f(y)-\lambda'y<k<f(y_0)-\lambda'y_0.
\end{align*}
The continuity of $f$ and the above inequalities imply that there must exist numbers $a,b$, and $c$ in $(x_0,\1y),(\1y,y_0)$, and $(y_0,y)$ respectively such that 
\[
f(a)-\lambda'a=f(b)-\lambda'b=f(c)-\lambda'c\ =k
\]
 which implies that $(a,f(a)),(b,f(b))$, and $(c,f(c))$ are colinear, a contradiction, and therefore $S(x_0,y)$ has to be greater than $\lambda'$ for $y>y_0$. Working similarly, we see that $S(x,y_0)<\lambda'$ for $x<x_0$.

An identical argument gives us that $\1y$ is the only point in $[x_0,y_0]$, and eventually in $[x_0,1)$, where $f(x)-\lambda'x$ attains a local minimum (see Figure~\ref{unique_min}) and from Lemma \ref{mono} we deduce that $f(x)-\lambda'x$ has to be monotone increasing in $[\1y,1)$. Hence, for any $x,y\geq\1y$ we have:
\[
x<y\iff f(x)-\lambda'x<f(y)-\lambda'y\xiff{x<y} S(x,y)>\lambda'.
\]

\begin{figure}%[h!]
	\centering
	\includegraphics[scale=0.2]{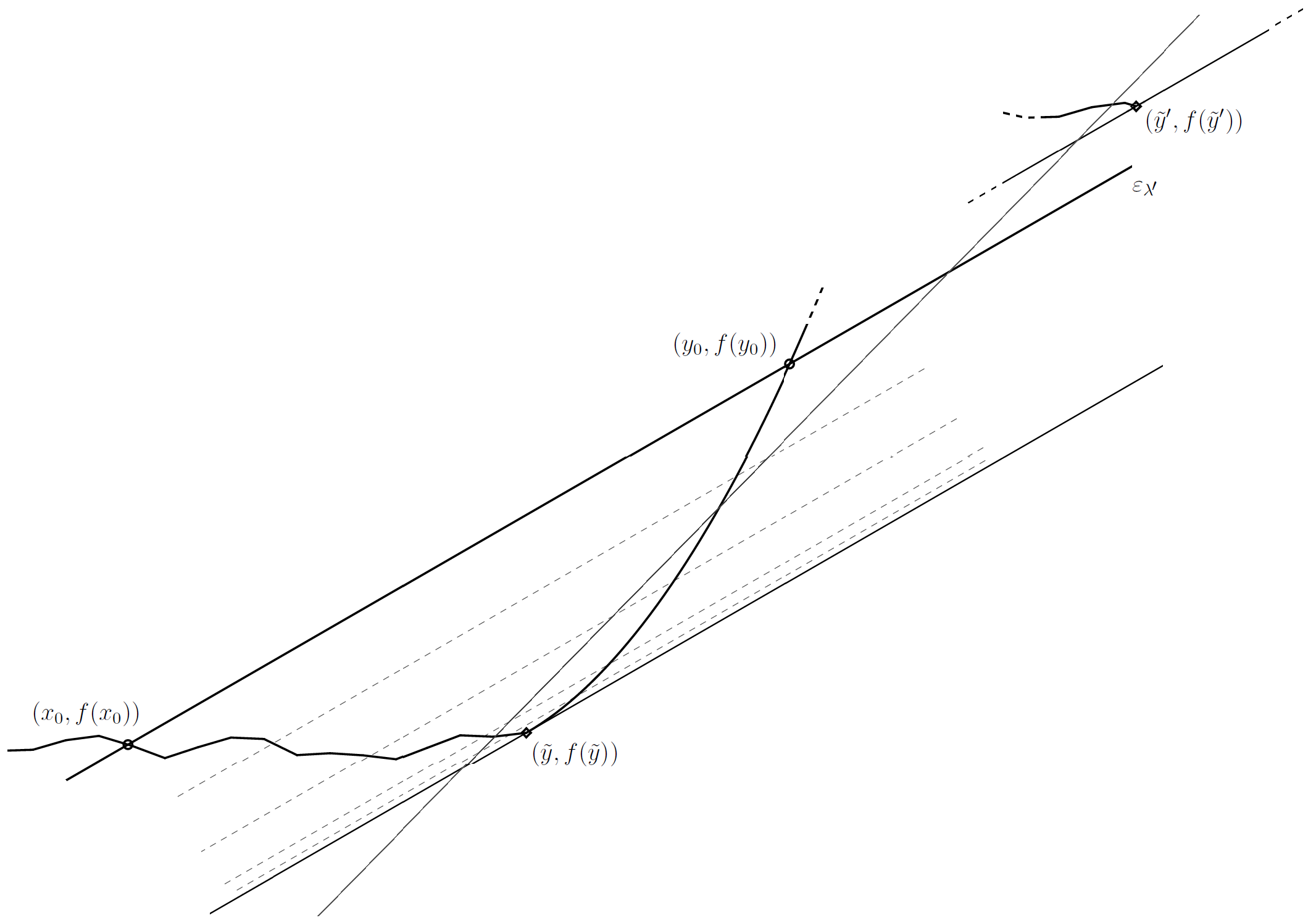}
	\caption{If $f$ attains a local minimum at another point $\1y'>\1y$, we can find a line of slope greater than $\lambda'$ intersecting $f$ at three points.}
	\label{unique_min}
\end{figure}

However, observe that for any $x$ and $y$ for which $S(x,y)>\lambda'$, the function $S(x,\var)$ has to be 1-1 otherwise our hypothesis fails in a similar way as above and, since it is continuous, it has to be monotone in $(x,1)$ for every $x\in[\1y,1)$. Therefore, $f$ is either convex or concave in $[\1y,1)$ and thus locally Lipschitz in $(\1y,1)$ thanks to Lemma \ref{lip}.

In particular, $f$ has to be convex in $[\1y,1)$. Indeed, assume $f$ is concave and let~$x$ be any number in $(\1y,y_0)$, see Figure~\ref{no_concave}. By concavity, the point $(\1y,f(\1y))$ has to lie below the line passing through $(y_0,f(y_0))$ with slope $\zeta=S(x,y_0)$ and, since $\zeta=S(x,y_0)>S(x_0,y_0)=\lambda'\geq \lambda$, the point $(x_0,f(x_0))$ lies above. Hence, this line will intersect the graph of $f$ at some point $(c,f(c))$ with $c\in(x_0,\1y)$ and the points $(c,f(c)),(x,f(x))$, and $(y_0,f(y_0))$ are colinear, a contradiction.

\begin{figure}%[h!]
	\centering
	\includegraphics[scale=0.15]{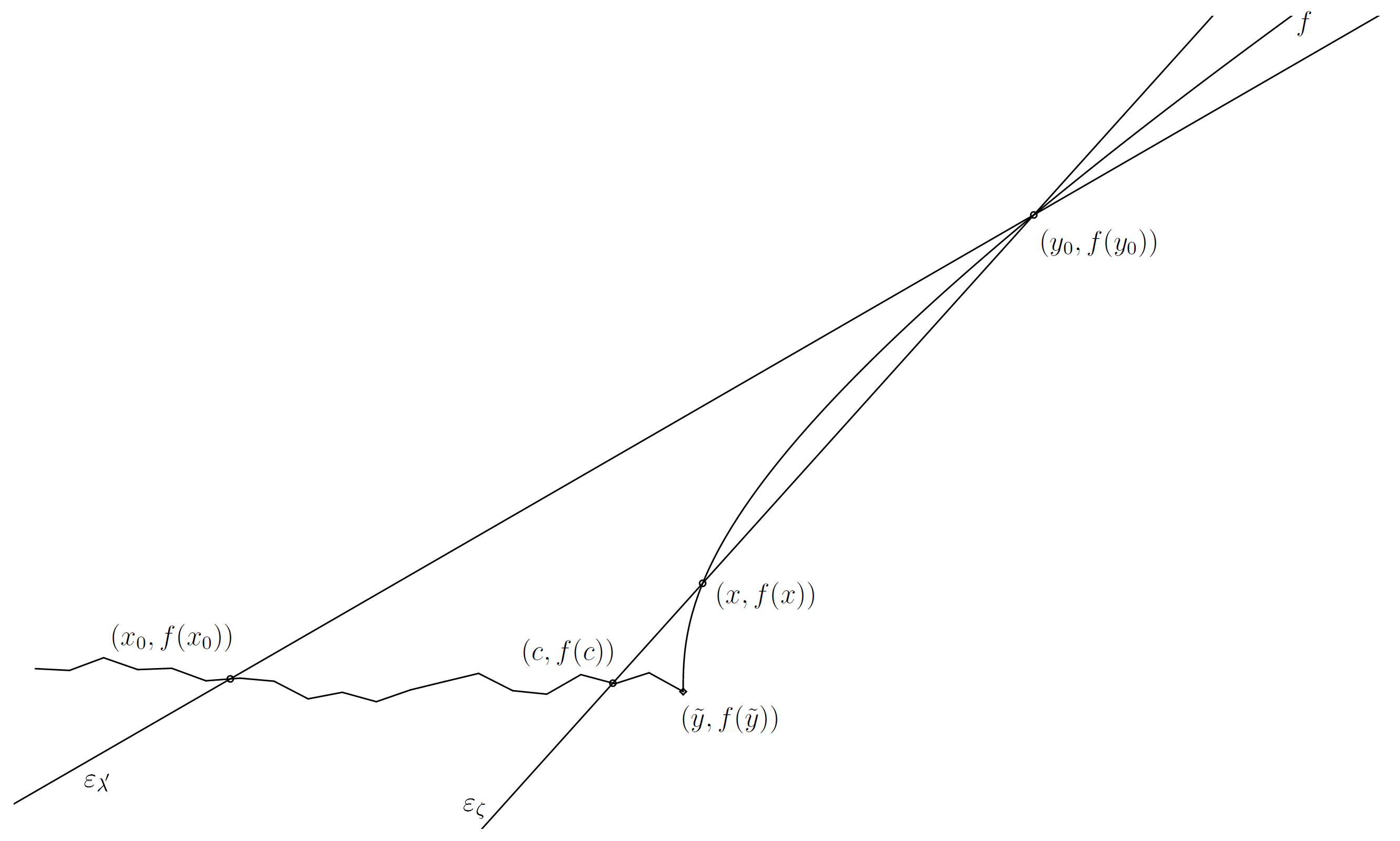}
	\caption{$S(x_0,y)$ has to be strictly monotone increasing in $(y_0,1)$.}
	\label{no_concave}
\end{figure}

If we instead assume that $S(x_0,y)>\lambda'$ for $x_0<y<y_0$, working similarly we conclude that there must exist $\1y\in[x_0,y_0]$ such that $f$ is concave in $(0,\1y]$.

The case when there exist $x_0,y_0\in(0,1)$ for which $S(x_0,y_0)=\lambda'\leq -\lambda$ is identical and gives us the reverse implications.

To sum up, we conclude that there are points $\1x,\1y\in(0,1)$ such that $f$ has some particular convexity on $(0,\1x]$ and on $[\1y,1)$. These intervals cannot overlap, because otherwise $f$ would be a line segment of slope at least $\lambda$ (or at most $-\lambda$) on $[\1y,\1x]$, which contradicts our hypothesis and so $\1x\leq\1y$. Let $\1 x$ be the maximal point so that $f$ is, for instance, convex on $(0,\1 x]$, and $\1 y$ the minimal so that $f$ is convex on $[\1 y,1)$. When $\1x\neq\1y$, for every points $x,y\in[\1x,\1y]$ we have $|S(x,y)|\leq \lambda$ and $f$ is Lipschitz in $[\1x,\1y]$ with Lipschitz constant $\lambda$.

This concludes the proof.
\end{proof}

Of course, any continuous function that satisfies the condition of the proposition and has different convexity on $(a,\1x]$ and on $[\1y,b)$ has to additionally satisfy $\lim_{x\to a^+,y\to b^-}|S(x,y)|<\lambda$.

Furthermore, notice that the fact that the cone is vertical (or at least that it contains the vertical line) is essential to get the locally Lipschitz property. Indeed, if~$C$ is a cone avoiding the vertical line, we can restrict the function~$\sqrt[3]{x}$ to a sufficiently small interval around $0$ so that it intersects all the lines of the cone at at most two points. But $\sqrt[3]{x}$ is clearly not Lipschitz around 0. However, we do have the following corollary.

\begin{figure}%[h!]
	\centering
	\begin{minipage}{0.49\linewidth}
		\frame
		{\includegraphics[width=\linewidth, height=3cm]{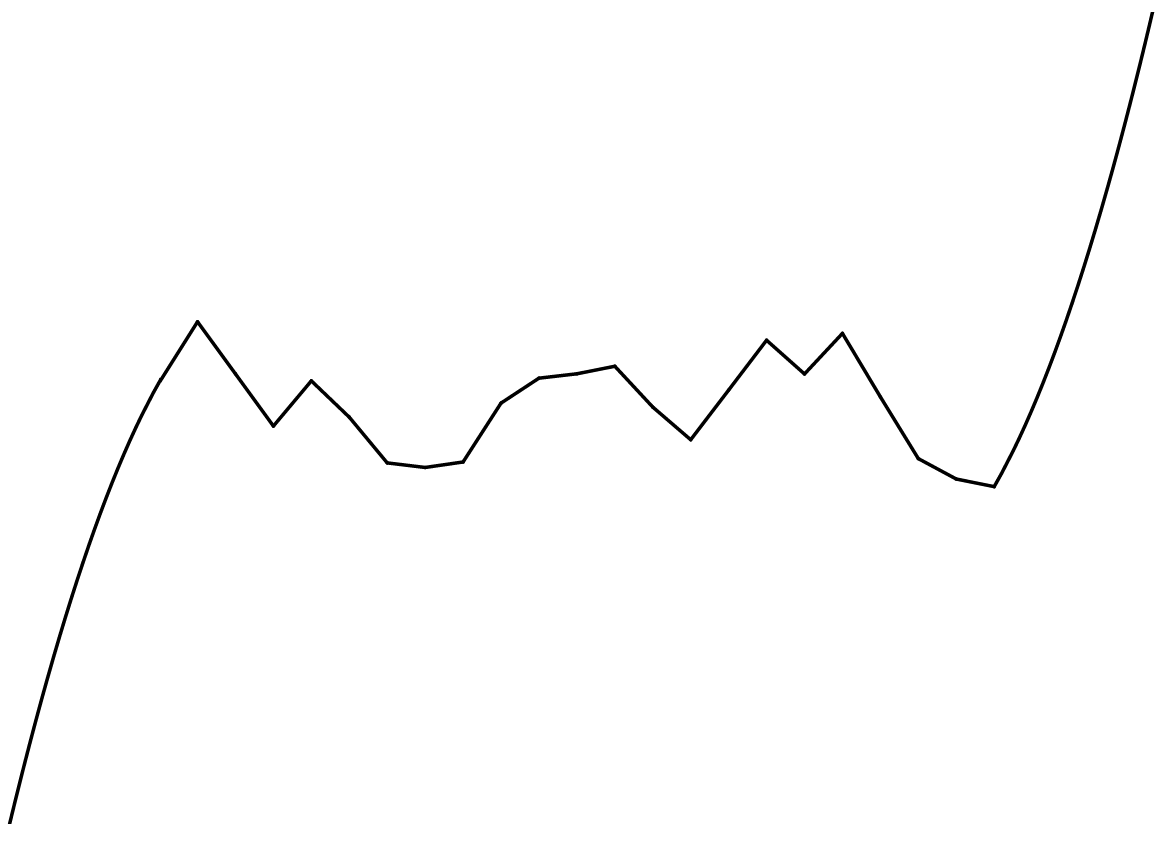}}
	\end{minipage}
	\begin{minipage}{0.49\linewidth}
		\frame
		{\includegraphics[width=\linewidth, height=3cm]{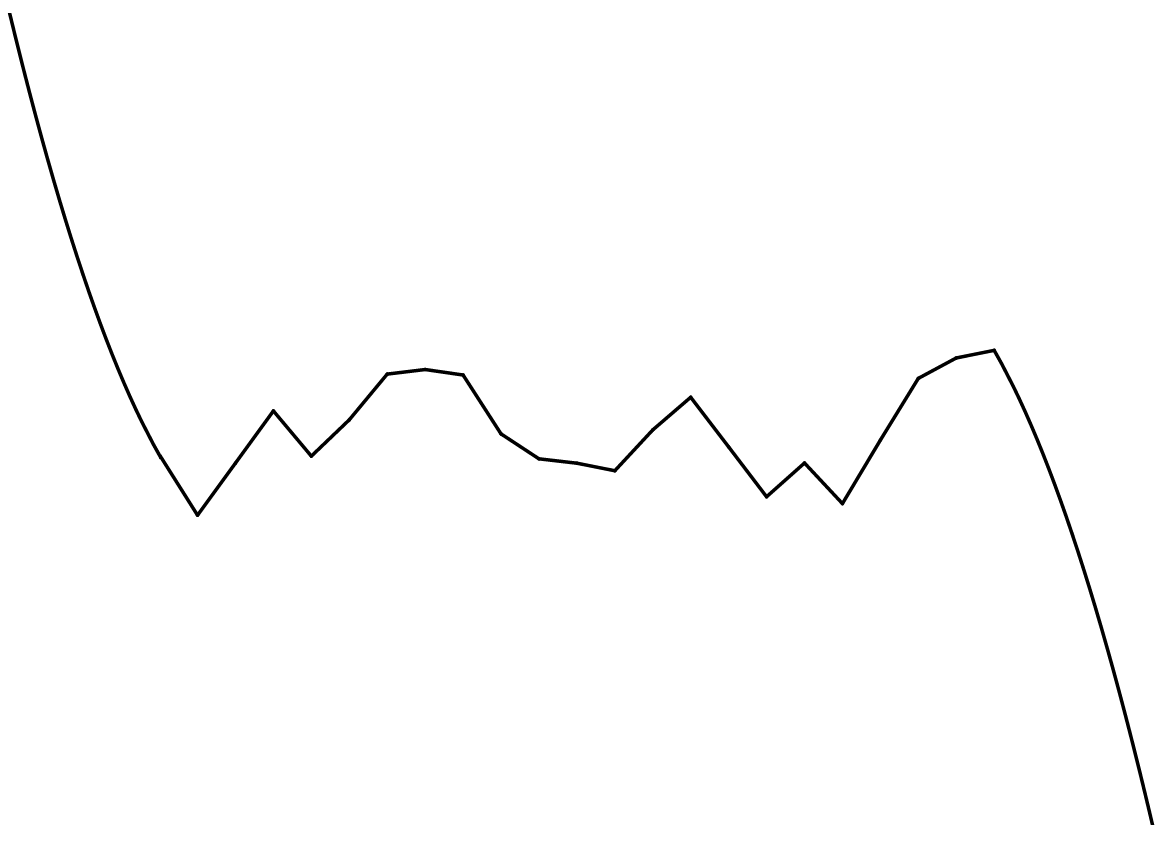}}
	\end{minipage}
	\begin{minipage}{0.49\linewidth}
		\frame
		{\includegraphics[width=\linewidth, height=3cm]{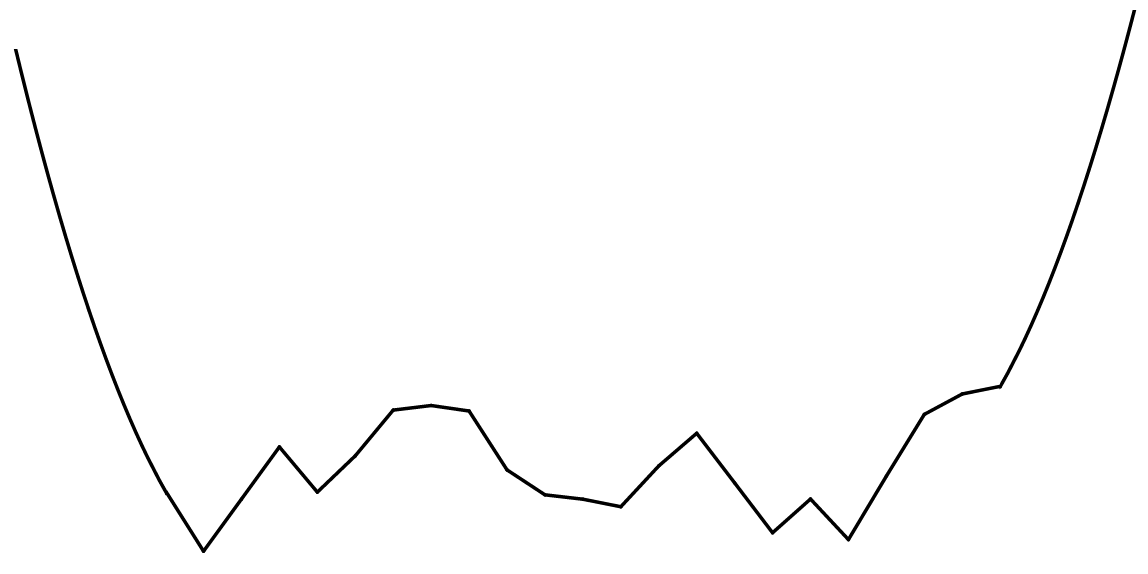}}
	\end{minipage}
	\begin{minipage}{0.49\linewidth}
		\frame
		{\includegraphics[width=\linewidth, height=3cm]{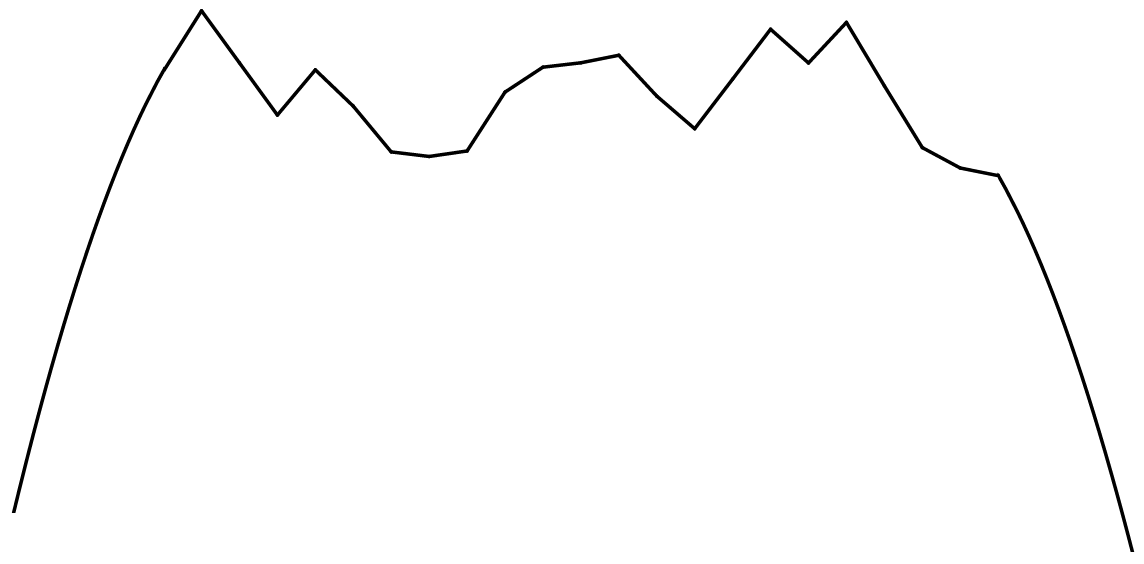}}
	\end{minipage}
	\caption{All the possible ways the graph of $f$ can look like.}
	\label{possibilities}
\end{figure}

\begin{cor*}
	Let $\lambda_1>0>\lambda_2$ be some fixed numbers and consider all the cones of lines with slopes between $\lambda_1$ and $\lambda_2$ (containing the vertical line). If $f\colon(0,1)\to\R$ is a continuous function satisfying the same condition as above, then it is locally Lipschitz.
\end{cor*}

\begin{proof}
	The inequalities $|S(x,y)|<\lambda$ and $|S(x,y)|\geq \lambda$ in this case correspond to $\lambda_2<S(x,y)<\lambda_1$ and $S(x,y)\geq \lambda_1\ or\ S(x,y)\leq \lambda_2$, respectively. The proof is the same as before and on the regions where $f$ is not convex or concave it is Lipschitz with Lipschitz constant the maximum of $\lambda_1$ and $-\lambda_2$.
\end{proof}

\begin{rem*}
	All the above remains true for any interval $(a,b)$. It is not hard to see that the same proof also works in the case where $f$ is defined on a closed interval, but Lemma \ref{lip} cannot be used in this setting. However, if $f\colon[0,1]\to\R$, its restriction $f\2{(0,1)}$ is locally Lipschitz.
\end{rem*}

\section{An example}

It is natural then to ask whether our assumption still gives us the locally Lipschitz property when we allow more points of intersection. It turns out this fails even for at most 3 points of intersection in the sense that there can be infinitely many points around where the function cannot be locally Lipschitz. Here, we construct such a function whose graph intersects a certain cone of lines at at most three points.

\bigskip

Consider the sequence $a_k=\frac{1}{2}-\frac{1}{2^k}$ for $k\geq1$, and on the each of the intervals $[a_k,a_{k+1}]$ define a continuous function $f_k$ with the following properties:
\begin{enumerate}[i)]
	\item $f_1(0)=0$, $f_1(\frac{1}{4})=f_2(\frac{1}{4})=\frac{\lambda}{4}$;
	\item $f_{k+1}(a_{k+1})=f_k(a_{k+1})$; \label{contin}
	\item $f_k(a_{k+1})=\frac{1}{2}\left(f_k(a_k)+f_{k-1}(a_{k-1})\right)$; \label{aver}
	\item $f_{2k}$ is monotone decreasing and convex on $[a_{2k},a_{2k+1}]$ and $f_{2k-1}$ is monotone increasing and concave on $[a_{2k-1},a_{2k}]$; \label{concav}
	\item the tangent line to $f_k$ at $(a_k,f_k(a_k))$ is vertical.
\end{enumerate}

Let $f\colon[0,1]\to\R$ be the function given by
\[
f(x)=\begin{cases}
f_k(x) & \afterline{if} x\in[a_k,a_{k+1}),\\
f_k(1-x) & \afterline{if} x\in(1-a_{k+1},1-a_k],\\
\frac{\lambda}{6} & \afterline{if} x=\frac{1}{2}
\end{cases}
\]
for all $k\geq 1$ (Figure~\ref{bat}), which is clearly continuous in $(0,1)\setminus\{\frac{1}{2}\}$ because of \eqref{contin}. Observe that the sequence $(b_k)=(f_k(a_k))$ is recursively defined by $b_{k+1}=\frac{b_k+b_{k-1}}{2}$ (through property \eqref{aver}) and it converges. In particular, we have $\frac{b_{k+1}-b_k}{b_k-b_{k-1}}=-\frac{1}{2}$ and therefore
\begin{equation}\label{sequen}
b_{k+1}=b_k+\Big(\frac{-1}{2}\Big)^{k-1}(b_2-b_1)\implies b_{k+1}=b_2-\frac{1}{3}\bigg(1-\Big(\frac{-1}{2}\Big)^{k-1}\bigg)(b_2-b_1).
\end{equation}
In our case, we have $b_1=f_1(0)=0$, $b_2=f_2(\frac{1}{4})=\frac{\lambda}{4}$, and also 
\[
{f_k(a_k)=\frac{\lambda}{6}\bigg(1-\Big(\frac{-1}{2}\Big)^{k-1}\bigg)},
\]
 hence $\lim_{k\to+\infty}f_k(a_k)=\frac{\lambda}{6}$. But note that for every $x\in(0,\frac{1}{2})$ there is an $n\geq1$ for which $x\in[a_n,a_{n+1})$ and, since each $f_k$ is monotone in $[a_k,a_{k+1})$ for every~$k$, we get 
\[
\min\big\{f_n(a_n),f_{n+1}(a_{n+1})\big\}\leq f(x)\leq\max\big\{f_n(a_n),f_{n+1}(a_{n+1})\big\}.
\]
Therefore, we have $\lim_{x\to\frac{1}{2}^-}f(x)=\frac{\lambda}{6}=f(\frac{1}{2})$, and similarly for $x\in(\frac{1}{2},1)$, which means that $f$ is also continuous at $\frac{1}{2}$.

However, by construction $f$ is locally Lipschitz on $(0,1)\setminus\{\frac{1}{2}\}$ except at around $a_k$ and $1-a_k$, $k\geq1$, and therefore it is not locally Lipschitz around $\frac{1}{2}$ either, because $a_k\to\frac{1}{2}$ as $k\to+\infty$.

\begin{figure}%[h!]
	\centering
	\includegraphics[width=\textwidth]{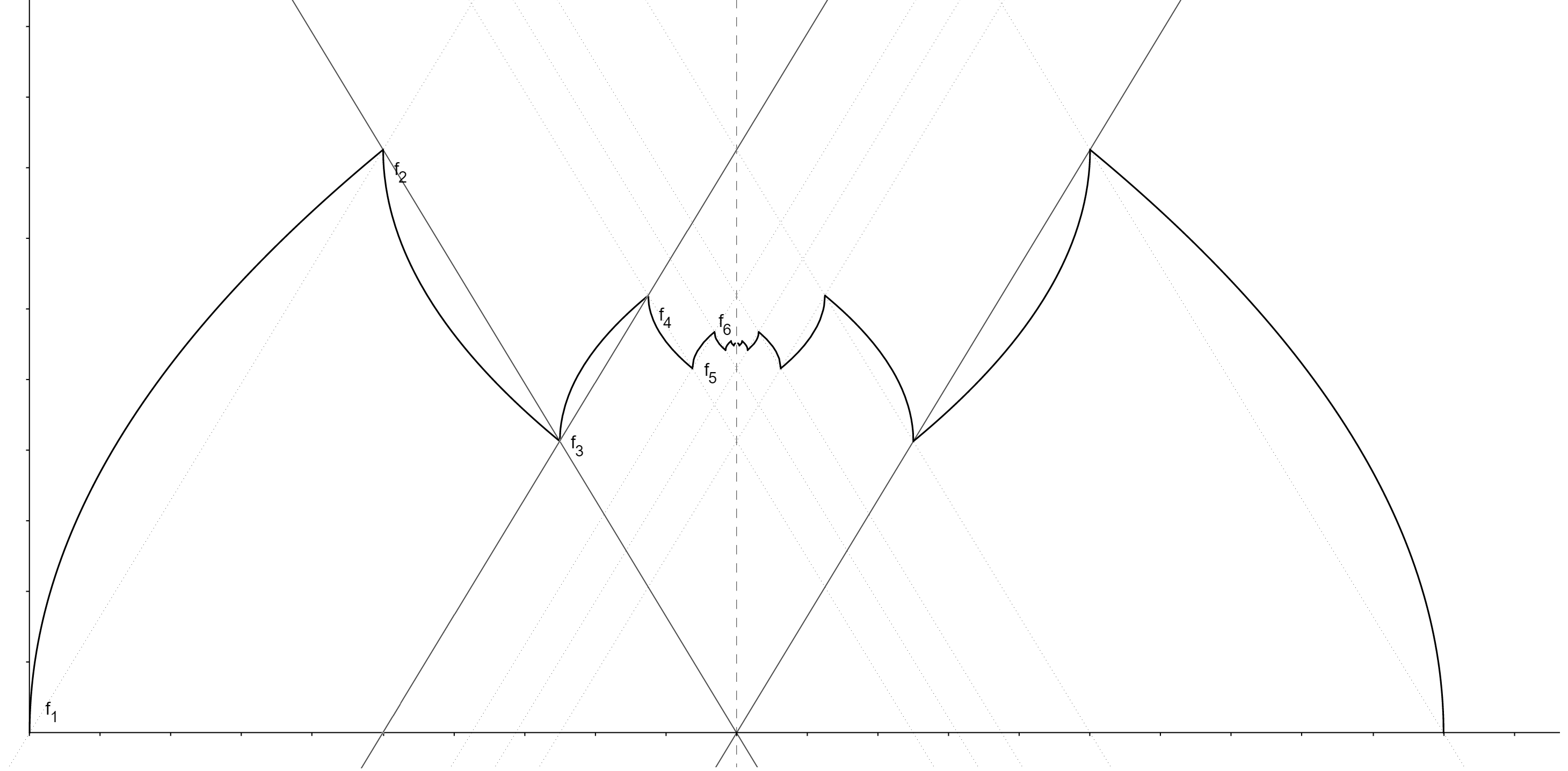}
	\caption{At most $3$ points of intersection with any line inside the cones}
	\label{bat}
\end{figure}

Now we proceed to show the graph of $f$ has at most 3 intersection points with any line inside a vertical cone with slopes between $\lambda$ and $-\lambda$.

Each $f_k$ is monotone and has certain concavity on $[a_k,a_{k+1}]$, hence its graph is contained inside the triangle $T_k$ with vertices $(a_k,f(a_k))$, $(a_{k+1},f_{k+1}(a_{k+1}))$, and $(a_k,f(a_{k+1}))$ (see Figure~\ref{triangles}) and therefore any line intersecting the graph of $f$ (at at least two points) has to pass through some of these triangles. Notice, however, that if a line passes through two nonconsecutive triangles, say $T_k$ and $T_{k+j}$ $(j>1)$, then it falls outside the admissible cone of lines. In particular, (because of properties \eqref{contin} through \eqref{concav}) each $T_{k+1}$ is half the size of $T_k$ and they are placed is such a way that the maximum and minimum slope a line through them can have are respectively the maximum and the minimum of the quantities
\[\frac{f_{k+j}(a_{k+j})-f_k(a_k)}{a_{k+j}-a_k}\inline{and}\frac{f_{k+j}(a_{k+j+1})-f_k(a_{k+1})}{a_{k+j}-a_{k+1}},\]
when one of the numbers $k$ and $k+j$ is even and the other is odd, and the maximum and minimum of the quantities
\[\frac{f_{k+j}(a_{k+j+1})-f_k(a_k)}{a_{k+j}-a_k}\inline{and}\frac{f_{k+j}(a_{k+j})-f_k(a_{k+1})}{a_{k+j}-a_{k+1}},\]
when $k$ and $k+j$ are both even or both odd. Using \eqref{sequen} we can see that each of the above is bounded in absolute value by $\lambda$ whenever $j>1$.

For the same reasons any admissible line passing through $(\frac{1}{2},\frac{\lambda}{6})$ intersects the graph only at that point, because
\[
\left|\frac{f_k(a_k)-\frac{\lambda}{6}}{a_k-\frac{1}{2}}\right|=\frac{\lambda}{3}<\lambda.
\]

Therefore, the admissible lines intersecting the graph necessarily pass through two (or maybe only one) consecutive triangles and each such line intersects the graph of $f_k$ at at most two points because of \eqref{concav}. Furthermore, due to the difference in concavity of $f_k$ and $f_{k+1}$, a line cannot intersect both of their graphs at two points, because then it would need to have both negative and positive slope, which is absurd.

\begin{figure}%[h!]
	\centering
	\includegraphics[width=\textwidth]{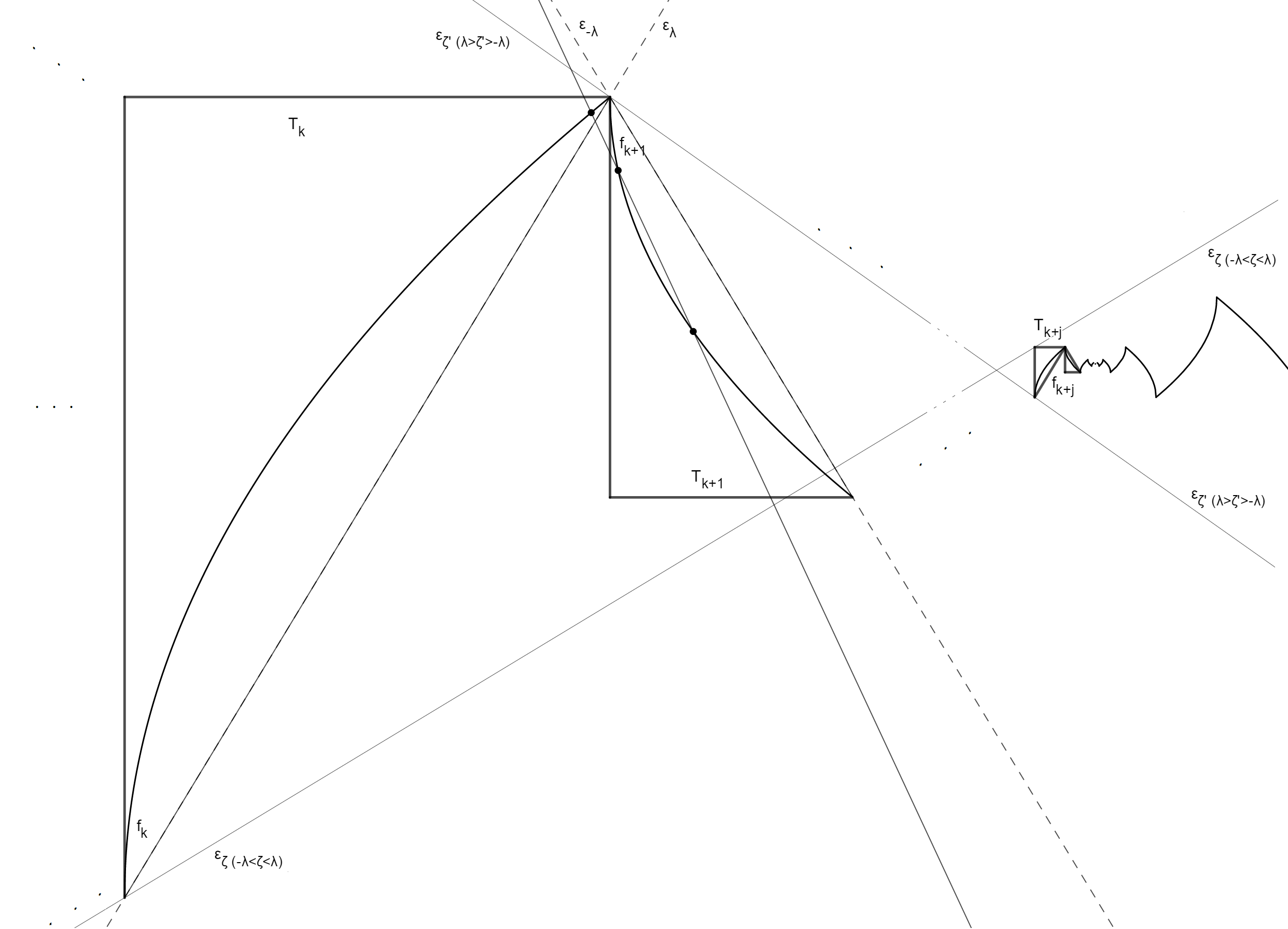}
	\caption{The case when $k$ and $k+j$ are both odd.}
	\label{triangles}
\end{figure}

An example of a sequence $(f_k)$ of functions with the above properties is the following:
\[f_k(x)=\frac{\lambda}{6}\bigg(1-\Big(\frac{-1}{2}\Big)^{k-1}\bigg)+\frac{(-1)^{k+1}\lambda}{2^{\frac{k+1}{2}}}\sqrt{x-a_k}.\]

\section{Hausdorff measure}

Marstrand in \cite[Theorem 6.5.III]{JM} proved that if a Borel set on the plane has the property that
\begin{tequation} \label{Mar}
	if the lines in a positive measure of directions intersect this Borel set at a set of Hausdorff dimension zero, then the Hausdorff dimension of this Borel set is at most~$1$.
\end{tequation}
In particular, this happens if the intersections are at most countable. The Borel assumption is essential.

That said, Marstrand's theorem does not in general guarantee the Hausdorff measure of the Borel set is finite. Our next goal will be to deal with the Hausdorff measure of a continuous curve and also generalise to arbitrarily many points of intersection with our cones (still finitely many, though). It turns out that the curve has to always be $\sigma$-finite with respect to the $\cH^1$ measure.

In order to proceed we need set up things more rigorously:
\begin{notat*}
	Let $C(\phi,0)=\{(x,y)\in\R^2\such|y|\geq \tan(\phi)\,|x|\}$ denote the vertical closed cone in between the lines through the point $(0,0)$ with slopes $\tan(\phi)$ and $-\tan(\phi)$ (where $0<\phi<\frac{\pi}{2}$). By $C_+(\phi,0)$ we will denote the upper half of the cone $C(\phi,0)$, that is $C_+(\phi,0)=\{(x,y)\in\R^2\such|y|\geq \tan(\phi)\,|x|,\ y\geq0\}$, and by $C_-(\phi,0)$ its lower half. Let $C(\phi,\rho)$ be the cone's counter-clockwise rotation by angle $\rho$, $C(\phi,0,h)=B_0(h)\cap C(\phi,0)$, where $B_x(r)=B(x,r)$ is the closed ball centred at $x$ with radius $r$, and $C_P(\phi,0)$ the translation of $C(\phi,0)$ so that its vertex is the point $P$. Finally, $C^*$ will denote the dual cone of $C$, that is $C^*(\phi,0)=\0{C(\phi,0)^C}$. We will be combining different notation in a natural way, for example $C_+(\phi,\rho,h)$ is the upper half of the truncated and rotated cone with vertex at $0$.
	
	$\gamma\colon[0,1]\to\R^2$ will be a continuous curve.
\end{notat*}

	\subsection{The main hypothesis}

\begin{tequation} \label{hypo}
	Fix an integer $k\geq 2$. Fix an angle $\phi\in(0,\frac{\pi}{2})$ and a rotation $\rho\in[0,2\pi)$. A line contained inside the cone $C_P(\phi,\rho)$ for any point $P\in\R^2$ intersects the curve $\gamma$ at at most $k$ points.
\end{tequation}

Any such line will be called \emph{admissible}. A cone consisting of only admissible lines will also be called \emph{admissible}.

	\subsection{\texorpdfstring{$\gamma$}{?} is \texorpdfstring{$\sigma$}{?}-finite}

For simplicity and without loss of generality we will assume the the curve $\gamma\colon[0,1]\to\R^2$ is bounded inside the unit square and that $(0,0),(1,1)\in \gamma$. We additionally assume that the cones of our hypothesis are vertical, i.e., that $\rho=0$.

\begin{theorem}
	$\gamma$ can be split into countably many sets $\gamma_n$ with finite $\cH^1$ measure. In particular, $\gamma$ is 1-rectifiable.
\end{theorem}

The following lemma plays a key role in the proof of this theorem, but we will postpone its proof until later.

\begin{lem} \label{free_of_cones}
	For every point $P\in \gamma$ there exists an admissible cone $C_P(\theta,\rho,h)$ that avoids the curve $\gamma$ except at $P,$ that is $C_P(\theta,\rho,h)\cap \gamma=\{P\}$.
\end{lem}

In view of Lemma \ref{free_of_cones} --- by slightly tilting $\rho$, enlarging $\theta$ and monotone decreasing $h$ --- we may assume the triplet $(\theta,\rho,h)$ consists of rational numbers. If $\{(\theta_n,\rho_n,h_n)\}$ is an enumeration of all rational triples that still lie within our admissible set, then we can decomposed $\gamma$ into the countably many sets
\[
\gamma_n=\big\{P\in \gamma\such C_P(\theta_n,\rho_n,h_n)\cap \gamma=\{P\}\big\}
\]
(see Figure~\ref{covering_balls}). Note that $\gamma_n$ are not necessarily disjoint for different values of~$n$.

We proceed to prove each one of them has finite $\cH^1$ measure. Note that this is not new knowledge and it can be found, for example, in \cite[Lemma 3.3.5]{F} or \cite[Lemma 15.13]{PMbook} in a more general setup. Nevertheless, we present it here for completeness. 

For the rest of this section $n$ will be fixed.

\begin{figure}%[h!]
	\centering
	\includegraphics[width=0.85\textwidth]{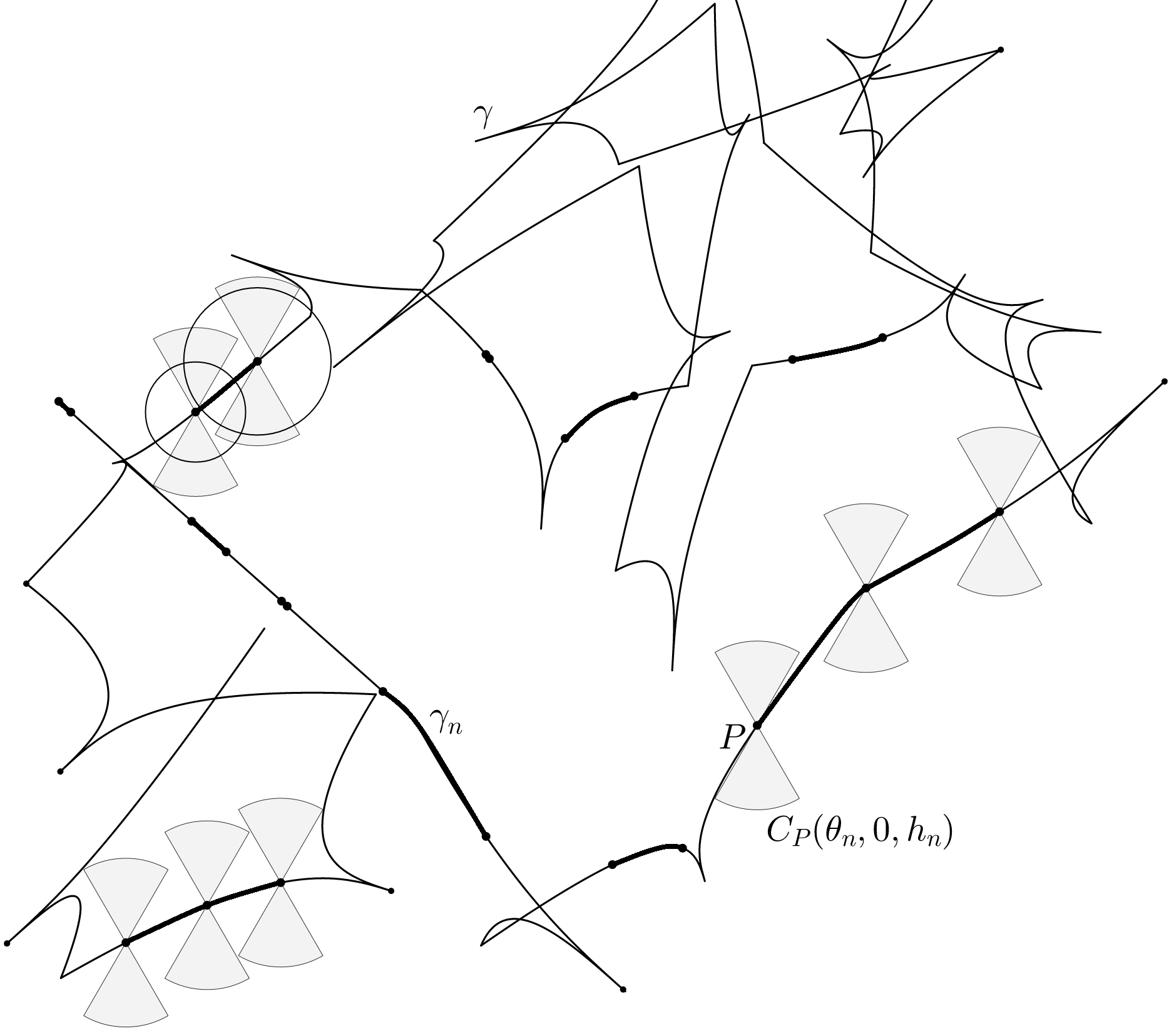}
	\caption{The curve $\gamma$ and its part $\gamma_n$ for $\theta_n$, $\rho_n=0$, and $h_n$.}
	\label{covering_balls}
\end{figure}

%\newpage
\begin{lem}
	$\cH^1(\gamma_n)<\frac{2k}{\cos(\theta_n)}$.
\end{lem}

\begin{proof}
Without loss of generality we may assume the cone $C_P(\theta_n,\rho_n,h_n)$ is vertical, i.e., that $\rho_n=0$. Let us now split the unit square into $N$ vertical strips, $S_j$ ($j=1,2,\dots,N$), of base length $\frac{1}{N}$ with $N$ sufficiently large so that $\frac{1}{N}<\cos(\theta_n)\,h_n$. Let $J$ be the set of indices $j$ for which
\[
S_j\cap \gamma_n\neq\emptyset
\]
and for any point $P\in \gamma$ denote the connected component of $\gamma$ inside $S_j$ through $P\in S_j\cap \gamma$ by $\Gamma^*_P(j)$.

Fix a $j\in J$ and consider a point $P\in S_j\cap \gamma_n$. Since $\frac{1}{N}<\cos(\theta_n)\,h_n$, the sides of $S_j$ necessarily intersect both sides of the cone $C_P(\theta_n,0,h_n)$ creating thus two triangles both contained inside the ball $B_P\big(\frac{1}{N\cos(\theta_n)}\big)$ (see Figure~\ref{the_cone}). For any point $P'\in S_j\cap \gamma_n$ other than $P$ there are two cases: either $|P-P'|\leq h_n$ or $|P-P'|> h_n$. In the first case, the sets $\Gamma^*_P(j)$ and $\Gamma^*_{P'}(j)$ are both contained inside the two triangles $C^*_P(\theta_n,0)\cap S_j$. In the second, they are necessarily disjoint, because $C_P(\theta_n,0,h_n)$ is free from points of $\gamma$ (other than~$P$). These additionally imply that there can be no more than $\frac{1}{\sin(\theta_n)\,h_n}$ such distinct paths inside $S_j$. In particular,
\[
P\in \Gamma^*_P(j) \subset S_j\cap \gamma\cap B_P(h_n)\subset C^*_P(\theta_n,0,h_n)\cap S_j\subset B_P\left(\tfrac{1}{N\cos(\theta_n)}\right).
\]

\begin{figure}%[h!]
	\centering
	\includegraphics[width=0.85\textwidth]{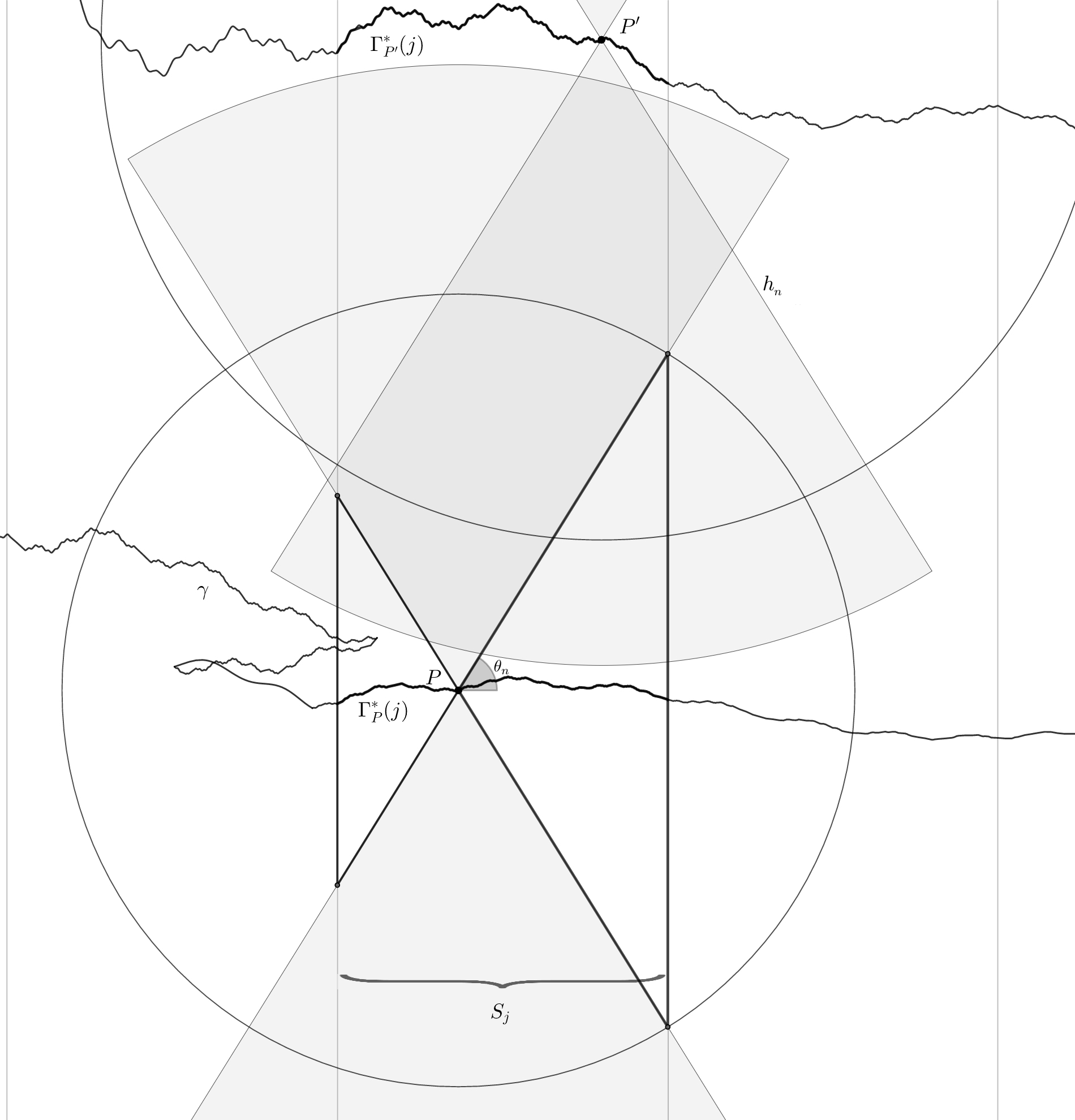}
	\caption{Each cone intersects a strip of length $\frac{1}{N}<\cos(\theta_n)\,h_n$.}
	\label{the_cone}
\end{figure}

Now, let $\mathcal{P}_j$ be a maximal set of points in $S_j\cap \gamma_n$ such that the sets $\Gamma^*_P(j)$ for %\afterline{for}
$P\in\mathcal{P}_j$
are all disjoint and observe that $S_j\cap \gamma_n$ is covered by the balls $B_P\big(\tfrac{1}{N\cos(\theta_n)}\big)$ with $P\in\mathcal{P}_j$. Indeed, if $P_0\in S_j\cap \gamma_n$ is not inside the set $\bigcup_{P\in\mathcal{P}_j}B_P\big(\tfrac{1}{N\cos(\theta_n)}\big)$, then by construction it is also outside $\bigcup_{P\in\mathcal{P}_j}B_P(h_n)$ and therefore $\Gamma^*_{P_0}(j)$ and $\Gamma^*_P(j)$ are disjoint for all $P\in\mathcal{P}_j$, which contradicts the maximality of $\mathcal{P}_j$. Moreover, due to the connectedness of $\gamma$, the set $\{P\}$ has to be path-connected with $(0,0)$ and $(1,1)$ and therefore each $\Gamma^*_P(j)$ has to intersect at least one side of the strip $S_j$. Hence, because of \eqref{hypo}, there can be at most $2k$ of these paths, i.e., $\#(\mathcal{P}_j)\leq\min\{2k,\frac{1}{\sin(\theta_n)\,h_n}\}\leq 2k$ for every $j\in J$. Therefore,
\[
\gamma_n\cap S_j\subset\bigcup_{P\in\mathcal{P}_j} B_P\left(\tfrac{1}{N\cos(\theta_n)}\right)\implies \gamma_n\subset\bigcup_{j\in J}\bigcup_{P\in\mathcal{P}_j}B_P\left(\tfrac{1}{N\cos(\theta_n)}\right)
\]
and the total sum of the radii of these balls is at most
\[
2k \frac{1}{N\cos(\theta_n)}\#(J)\leq\frac{2k}{\cos(\theta_n)}.
\]

Finally, if $\1{\gamma}_n=\{P\in \gamma\such C_P(\theta_n,0,h_n/2)\cap \gamma=\{P\}\}$, then $\gamma_n\subset\1{\gamma}_n$. Repeating the above construction with $\frac{1}{N}<\cos(\theta_n)\,\frac{h_n}{2}$, we get a cover of $\1{\gamma}_n$ --- and thus of $\gamma_n$ --- consisting of balls with a total sum of radii at most $\frac{2k}{\cos(\theta_n)}$. The result follows.
\end{proof}

\begin{rem*}
	In the above construction we are in fact able to cover the whole part of $\gamma$ inside $\bigcup_{j\in J}S_j$ with the same balls, and not merely $\gamma_n$.
\end{rem*}

Eventually, the curve $\gamma$ has to be $\sigma$-finite.

	\subsection{Cones free of \texorpdfstring{$\gamma$}{?}}

Here we prove Lemma \ref{free_of_cones}.

Fix $P\in \gamma$. Since $\gamma$ is bounded, there must exist an $\1 h>0$ such that $C_P(\phi,0)\cap \gamma=C_P(\phi,0,\1 h)\cap \gamma$. If
\[
C_P(\phi',0)\cap \gamma=\{P\}\inline{or}C_P(\phi',0,h)\cap \gamma=\{P\}
\]
for some $\phi'\in[\phi,\frac{\pi}{2})$ and some $h>0$, then we are done.

Suppose this does not happen. Then, for all $\phi'\in[\phi,\frac{\pi}{2})$ and for all sufficiently small $h>0$ we have
\begin{equation} \label{no_1cone}
C_P(\phi',0,h)\cap \gamma\setminus\{P\}\neq\emptyset.
\end{equation}

\begin{lem} \label{components}
	For any $P\in \gamma$ the set $C_P(\phi,0)\cap \gamma$ has finitely many (closed) connected components.
\end{lem}

\begin{proof}
Since $\gamma$ is connected, every point of $C_P(\phi,0)\cap \gamma$ has to be path-connected with the point $P$ through some part of the curve $\gamma$. There are two possibilities: either that path is entirely contained inside $C_P(\phi,0)$ or it has to pass through its sides. If a path does not intersect the sides, then it necessarily has to pass through $P$ otherwise $\gamma$ would not be connected. This yields precisely one connected component --- the one containing $P$ --- and all the rest (if any) have to intersect the sides of the cone.
If these components are infinitely many, there have to exist also infinitely many points of intersection on the sides of the cone; at least one for each connected component. But this contradicts~\eqref{hypo}.
\end{proof}

\begin{rem*}
	The connected components of Lemma \ref{components} total at most $2k$ and $P$ need not be a point of the curve. This lemma is still valid regardless of the cone we are working with as soon as it is in our admissible family of cones.
\end{rem*}

Let $\Gamma_P(\phi,0)$ be the connected component of $C_P(\phi,0)\cap \gamma$ that contains the point $P$, which because of \eqref{no_1cone} cannot be precisely the point set $\{P\}$. Because of Lemma \ref{components}, the set $C_P(\phi,0)\cap \gamma\setminus \Gamma_P(\phi,0)$ is compact and thus there exists $h_0>0$ such that $C_P(\phi,0,h_0)\cap \gamma\subset \Gamma_P(\phi,0)$. Observe that $C_P(\phi,0)\cap \gamma\setminus \Gamma_P(\phi,0)$ could be empty in general in which case $h_0=\infty$, however, we can always assume that $h_0\leq\1h$.

Next, we bisect our cone into two new identical cones sharing one common side
\[
C_P(\phi,0)=C_P(\phi_1,\rho_1)\cup C_P(\phi_1,-\rho_1),
\]
where $\phi_1=\frac{\pi}{4}+\frac{\phi}{2}$ and $\rho_1=\frac{\pi}{4}-\frac{\phi}{2}$, and repeat the above arguments for each new cone: If
\[
C_P(\phi',\rho_1)\cap \gamma=\{P\}\inline{or}C_P(\phi',\rho_1,h)\cap \gamma=\{P\}
\]
for some $\phi'\in[\phi_1,\frac{\pi}{2})$ and some $h>0$, then we are done. Similarly for $-\rho_1$ in place of $\rho_1$.

Suppose none of these happen. Then, for all $\phi'\in[\phi_1,\frac{\pi}{2})$ and for all sufficiently small~$h$ and~$h'$ we have
\begin{equation} \label{no_2cone}
C_P(\phi',\rho_1,h)\cap \gamma\setminus\{P\}\neq\emptyset\inline{and}C_P(\phi',-\rho_1,h')\cap \gamma\setminus\{P\}\neq\emptyset.
\end{equation}
We denote by $\Gamma_P(\phi_1,\rho_1)$ and $\Gamma_P(\phi_1,-\rho_1)$ the connected component of 
\[
C_P(\phi_1,\rho_1)\cap \gamma
\inline{and} 
C_P(\phi_1,-\rho_1)\cap \gamma
\]
 containing $P$, respectively. Then, the sets $C_P(\phi_1,\rho_1)\cap \gamma\setminus \Gamma_P(\phi_1,\rho_1)$ and $C_P(\phi_1,-\rho_1)\cap \gamma\setminus \Gamma_P(\phi_1,-\rho_1)$ are compact (thanks to Lemma \ref{components}) and thus there exist $h_{1,0},h_{1,1}\in(0,\1 h]$ such that $C_P(\phi_1,\rho_1,h_{1,0})\cap \gamma\subset \Gamma_P(\phi_1,\rho_1)$ and $C_P(\phi_1,-\rho_1,h_{1,1})\cap \gamma\subset \Gamma_P(\phi_1,-\rho_1)$.

\begin{figure}%[h!]
	\centering
	\includegraphics[draft=false,width=0.85\textwidth]{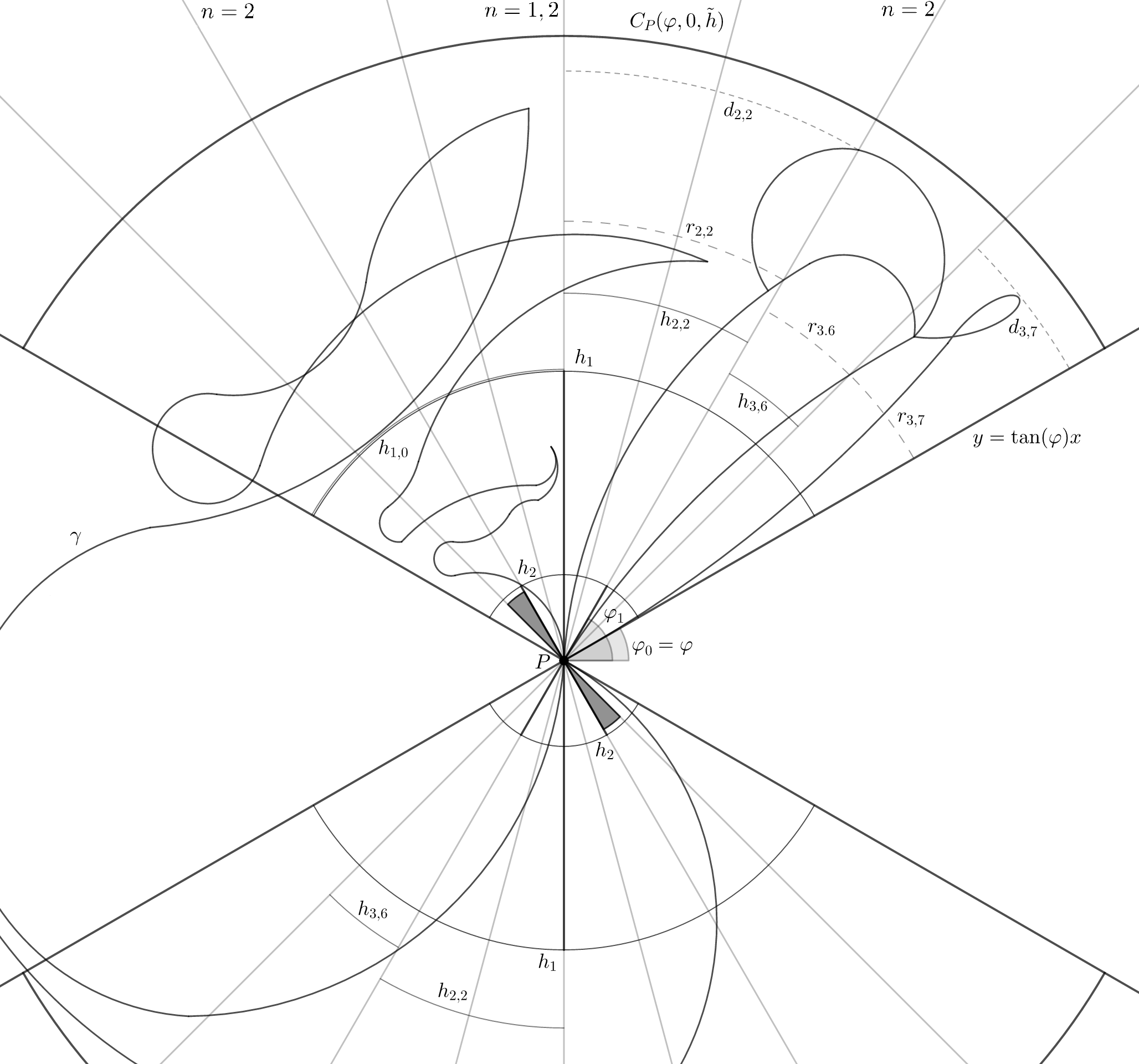}
	\caption{Finding a cone free from points of $\gamma$. The parameters $r$, $d$, and $h$ determine the radius.}
	\label{avoiding_cone}
\end{figure}

We iterate this construction indefinitely (Figure~\ref{avoiding_cone}). If at any step we get
\begin{equation} \label{cone_find}
C_P(\phi',\rho,h)\cap \gamma=\{P\}
\end{equation}
for some $\phi'$, $\rho$, and $h$, then we have found our desired cone and we stop. Otherwise, we get an infinite sequence of smaller and smaller cones satisfying the following:
\begin{align*}
\{P\}\varsubsetneq C_P(\phi_n,\rho_{n,i},h_{n,i})\cap \gamma\subset \Gamma_P(\phi_n,\rho_{n,i})\subset C_P(\phi_n,\rho_{n,i})&
\\
\mbox{for all }i=0,1,\dots,2^n-1&
\end{align*}
for all $n\geq 0$ where
\begin{align*}
& \phi_0=\phi && \phi_1=\frac{\pi}{4}+\frac{\phi}{2} && && \phi_n=\frac{\pi}{4}+\frac{\phi_{n-1}}{2}\\
& \rho_{0,0}=0 && \rho_{1,0}=\rho_1=\frac{\pi}{4}-\frac{\phi}{2} && \rho_{1,1}=-\rho_1 && \rho_{n,i}=(\phi_n-\phi)-i\frac{2(\phi_n-\phi)}{2^n-1}\\
& h_{0,0}=h_0 && && && 0<h_{n,i}\leq\1 h.
\end{align*}
Note that at the $n$th iteration we have exactly $2^n$ truncated closed cones separated by the lines
\[
l_{n,i}=P+\big\{(x,y)\such y=\tan(\pi-\phi_n+\rho_{n,i})\,x\big\}
\]
through $P$. The sets $\Gamma_P(\phi_n,\rho_{n,i})$ might intersect these lines, but this can happen at at most $k$ may points due to \eqref{hypo}. Let $r_{n,i}$ be the smallest distance between these points of intersection (if any) and $P$, that is
\[r_{n,i}=\dist\big(P,l_{n,i}\cap \Gamma_P(\phi_n,\rho_{n,i})\setminus\{P\}\big)\]
(again we can arbitrarily set some $0<r_{n,i}\leq\1 h$ if $l_{n,i}\cap \Gamma_P(\phi_n,\rho_{n,i})\setminus\{P\}=\emptyset$) and let
\[d_{n,i}=\min\Big\{\sup\{d(P,\Gamma_{P+}(t)\setminus P)\such t\in(0,1]\},\ \sup\{d(P,\Gamma_{P-}(t)\setminus P)\such t\in(0,1]\}\Big\}\]
where $\Gamma_{P+}(t)$ and $\Gamma_{P-}(t)$ are parametrisations of $\Gamma_P(\phi_n,\rho_{n,i})\cap C_{P+}(\phi_n,\rho_{n,i})$ and $\Gamma_P(\phi_n,\rho_{n,i})\cap C_{P-}(\phi_n,\rho_{n,i})$ respectively (which in general could be precisely the point set $\{P\}$) with $\Gamma_{P+}(0)=\Gamma_{P-}(0)=P$. Finally, we set
\[h_n=\min\{r_{n,i},\ d_{n,i},\ h_{n,i}\such i=0,1,\dots,2^n-1\}.\]
Since the above set is finite, $h_n>0$. From this construction for every $n\geq 0$ we get a collection of truncated cones $C_P(\phi_n,\rho_{n,i},h_n)$, for $i=0,1,\dots,2^n-1$, (see Figure~\ref{avoiding_cone}) that have the following property.
\begin{tequation} \label{paths}
	There is a path (part of $\gamma$) lying inside the cone that connects the point $P$ with at least one of the two arcs of length $(\pi-2\phi_n)h_n$ which bound the cone $C_P(\phi_n,\rho_{n,i},h_n)$. Moreover, these paths avoid any other intersections with that cone's boundary aside $P$ and the (closed) arc(s).
\end{tequation}

Now, fix $n$ sufficiently large so that $2^n\geq 2k+3$. Then, we can find at least $k+2$ of the cones $C_P(\phi_n,\rho_{n,i},h_n)$ that contain some path of those mentioned at \eqref{paths} all lying on the same half-cone, say on $C_{P+}(\phi,0,h_n)$. Consider one of the sides of our initial cone $C_P(\phi,0)$, say $l=P+\{(x,y)\such y=\tan(\phi)\,x\}$, fix $0<\epsilon<h_n\sin(\pi-2\phi_n)$ and translate $l$ vertically by $\epsilon$: $l_\epsilon=l+(0,\epsilon)$. Then, $l_\epsilon$ necessarily intersects all the $2^n$ different sectors of the ball $B_P(h_n)$ inside $C_{P+}(\phi,0,h_n)$, but only the right-most one, $C_{P+}(\phi_n,\rho_{n,2^n-1},h_n)$, at its arc-like part of the boundary. In particular, $l_\epsilon$ has to intersect the sides of at least $k+1$ sectors that contain the paths described in \eqref{paths} and therefore also intersects these paths. Hence, $l_\epsilon$ is one of our admissible lines that has at least $k+1$ intersections with $\gamma$, a contradiction.

Lemma \ref{free_of_cones} is proved.
\hfill\qed

\begin{rems*}
	\begin{enumerate}[i)]
		\item In the definition of $h_n$, three different parameters occur, $r_{n,i}$, $d_{n,i}$, and $h_{n,i}$. Without $h_{n,i}$, \eqref{cone_find} automatically fails; $d_{n,i}$ is to ensure $\Gamma_P(\phi_n,\rho_{n,i})$ will always intersect the boundary of the corresponding cone and $r_{n,i}$ forces this intersection to avoid the sides.
		\item In the above construction we bisected the initial cone into $2$, $4$, $8$ etc. smaller cones every time. However, any possible way to cut the cones would still work as soon as it eventually yields an infinite sequence. 
		\item The same proof can be applied to any cone within our admissible set of directions.
	\end{enumerate}
\end{rems*}

\section{Higher dimensions}

Mattila in \cite[Lemma 6.4]{PM} generalised Marstrand's results from \cite{JM} and showed the following.
\begin{lem}[Mattila] \label{l:Mattila--Marstrand}
	Let $E$ be an $\cH^s$ measurable subset of $\R^n$ with $0<\cH^s(E)<\infty$. Then,
	\[\dim(E\cap (V+x))\geq s+m-n\]
	for almost all $(x,V)\in E\times G(n,m)$. 
\end{lem}
In particular, for a Borel set in, say, $\R^2$ we have:
\begin{tequation*}
	if any 2-dimensional plane in a positive measure of directions intersects this Borel set at a set of Hausdorff dimension at most 1, then the Hausdorff dimension of this Borel set is at most 2.
\end{tequation*}
Furthermore, if every line in the direction of some 2-dimensional cone intersects a Borel set (not merely the graph of some continuous function) at at most countably many points, then any 2-dimensional plane in a positive measure of directions intersects this Borel set by a set of Hausdorff dimension at most $1$ (Marstrand) and then the Hausdorff dimension of this Borel set is at most $2$ (Mattila).

Of course, the same is also true in $\R^n$, that is, if a Borel set has countable intersection with a certain cone of lines, then its dimension does not exceed $n-1$.

\smallskip

Now, we restrict our attention to what happens with only $2$ points of intersection in higher dimensions and we would like to generalize Proposition \ref{2D} to $\R^n$.

Suppose we have a continuous function $z=f(x,y)$, say, on a square in $\R^2$, satisfying the property that
\begin{tequation} \label{2pts}
	any line in the direction of a certain open cone with axis along a vector $\textbf{v}\in\R^3$ intersects the graph at at most two points.
\end{tequation}
Then, we would want $f$ to obey the same rule. Namely we ask the following:
\begin{quest*} \label{3D}
	Is a continuous function on $(-1,1)^2$ having property \eqref{2pts} locally Lipschitz?
\end{quest*}

\section{Relationships with perturbation theory}

The problem we consider in this note grew from a question in perturbation theory of self-adjoint operators (see \cite{LTV}). The question was to better understand the structure of Borel sets in $\R^n$ that have a small intersection with a whole cone of lines. Marstrand's and Mattila's theorems in \cite{JM} and \cite{PM}, respectively, give a lot of information about the exceptional set of finite-rank perturbations of a given self-adjoint operator. The exception happens when singular parts of unperturbed and perturbed operators are \emph{not mutually singular}. It is known that this is a rare event in the sense that its measure is zero among all finite-rank perturbations. The paper \cite{LTV} proves a stronger claim: the dimension of a bad set of perturbations actually drops.

Let us explain what was the thrust from \cite{LTV} and why that paper naturally gives rise to the questions considered above: what is the structure of Borel sets in $\R^n$ that have a small intersection with all the lines filling a whole cone and their parallel shifts? 

In \cite{LTV}, a family of finite rank (self-adjoint) perturbations, $A_\alpha$, of a self-adjoint (suppose bounded for simplicity) operator $A$ in a Hilbert space $\mathcal{H}$ is considered:
\[A_\alpha := A + B\alpha B^*\]
parametrized by self-adjoint operators $\alpha\colon \mathbb{C}^d\to \mathbb{C}^d$ (i.e., Hermitian matrices). The operator $B\colon\mathbb{C}^d\to \mathcal{H}$ is a fixed injective and bounded operator. It is also assumed that range of $B$ is cyclic with respect to $A$. In the case when $d=1$ (rank-one perturbations), the Aronszajn-Donoghue theorem states that the singular parts of the spectral measures of $A$ and $A_\alpha$ are always mutually singular. However, it is known that for $d>1$ the singular parts of the spectral measures of unperturbed and perturbed operators are not always mutually singular. 

Notice that the space of perturbations, that is the space $H(d)$ of Hermitian $(d\times d)$ matrices, has dimension $d^2$. In \cite{LT}, it was proved that, given a singular measure $\nu$, the scalar spectral measure $\mu_\alpha$ of the perturbation $A_\alpha$ is \emph{not} singular with respect to $\nu$ for the set of $\alpha$'s having zero Lebesgue measure in $H(d)$. Such $\alpha$'s are called \emph{exceptional}, and this result shows that even though the set of exceptional $\alpha$'s can be non-empty (for $d>1$), it is a thin set. But is it maybe thinner?

In fact, the following result was proved in \cite{LT}. Fix $\alpha_0, \alpha_1\in H(d)$ where $\alpha_1$ is in the cone of positive Hermitian matrices and consider $\alpha(t)= \alpha_0+t\alpha_1$. Then, for any such $\alpha_0,\alpha_1$ there are at most countably many $t\in \R$ such that the $\alpha(t)$ is exceptional. This extra information allowed the authors in \cite{LTV} to prove that the Hausdorff dimension of exceptional perturbations is actually at most~${d^2-1}$.

The reader might have noticed an underlying geometric measure theory fact: a Borel set in $\R^n$ (here $n=d^2$) that has an at most countable intersection with a whole cone of lines and their parallel shifts is, in fact, of dimension $n-1$.

Thus the dimension drop detected in Marstrand's and Mattila's theorems was instrumental for the drop in dimension for exceptional perturbations.

It seems enticing to understand the structure of the sets that have even less than countable intersection with all parallel shifts of all lines from a fixed cone. Suppose the Borel set under investigation intersects only at at most two, or at most $k<\infty$, points with these lines. What additional knowledge one can obtain about this set? This question motivated the work presented in the previous sections.

\end{document}